\documentclass[12pt]{amsart}

\usepackage{color}
\usepackage{mathrsfs}
\usepackage[colorlinks,citecolor=red,pagebackref,hypertexnames=false, breaklinks]{hyperref}

\newtheorem{theorem}{Theorem}[section]
\newtheorem{proposition}[theorem]{Proposition}
\newtheorem{coro}[theorem]{Corollary}
\newtheorem{lemma}[theorem]{Lemma}
\newtheorem{example}[theorem]{Example}
\newtheorem{definition}[theorem]{Definition}

\theoremstyle{remark} \newtheorem{remark}[theorem]{Remark}

\newcommand\N{\mathbb{N}}
\newcommand\R{\mathbb{R}}
\newcommand\Z{\mathbb{Z}}
\newcommand\F{\mathbb{F}}
\newcommand\M{\mathbb{M}}

\newcommand{\disappear}[1]

\begin{document}

\title[Quantum separation effect for Feynman-Kac semigroup]
{Probabilistic approach to quantum separation effect for Feynman-Kac semigroup}

\author{Adam Sikora}
\address{Adam Sikora, Department of Mathematics, Macquarie
University, NSW 2109, Australia}
\email{adam.sikora@mq.edu.au}

\author{Jacek Zienkiewicz}
\address{Jacek Zienkiewicz, Instytut Matematyczny, Uniwersytet Wroc\l awski,
  50-384 Wroc{\l}aw, pl. Grunwaldzki 2/4, Poland}
\email{zenek@math.uni.wroc.pl}

\thanks{Adam Sikora was partially supported by an 
Australian Research Council (ARC) Discovery Grant  DP130101302.
Jacek Zienkiewicz  was partially supported by  NCN
grant   UMO-2014/15/B/ST1/00060  }
\date{\today}

\begin{abstract}
The quantum tunnelling  phenomenon allows a particle in Schr\"odinger mechanics to tunnel through a barrier that it classically could not overcome. Even infinite potentials do not always form impenetrable barriers. We
discuss an answer to the following question: What is a critical magnitude of potential, which creates an impenetrable barrier and   for which the corresponding Schr\"odinger evolution system separates?  In addition we describe some quantitative estimates for the separating effect in terms of cut-off potentials. 
\end{abstract}

\maketitle

{\it In memoriam of our high school mathematics teacher Augustyn Ka{\l}u\.za }

\section{introduction}

The main motivation for our study comes from the notion of the quantum tunnelling effect, a phenomenon which illuminates a striking difference between classical and quantum mechanics.
It allows a microscopic particle to pass through the classically forbidden potential barrier, even if its height is infinite.  It is easily predicted and explained by Schr\"odinger mechanics - the eigenstates of the Hamiltonian of the system cannot be localised. In this work we consider a quantum well and investigate the possibility that a particle
trapped in a well cannot escape, that is the possibility that the barrier separates two regions.
To be more precise, we consider the domain ${D}$ and its boundary $K=\partial D$ that separates ${D}$ and its complement 
${D}^c$. Then we fix the special  class of potentials $V$ singular near $K$ and consider 
the Hammiltonian of the system, that is the operator
$$
L_V=\Delta -V
$$
initially defined for a function belonging to $C_c^\infty(\R^d\setminus K)$. Here $\Delta$ is the positive standard Laplace operator.
Then we consider the Feynman-Kac semigroup $\exp(-t L_V)$ generated by  the  Hamiltonian
$L_V$ and we denote by $p^V_t(x,y)$ the corresponding heat kernel. 
We address the question: {\em when does $\exp(-t L_V)$ separate
${D}$ and ${D}^c$, that is when is $p^V_t(x,y)=0$ for all  $x\in {D}$ and $y\in {D}^c$?}  

When the domain ${D}$ has a  smooth boundary, the problem considered by us  has been satisfactory   resolved by Wu in \cite{Wu}. In our  work we generalise the results obtained by Wu. The essential difference compared to   \cite{Wu} is that we do not require smoothness
of the considered domain $D$. The domains we consider are irregular  fractals  of some special but still general type including the Koch 
snowflake domain.
In addition, we study quantitative estimates of the  tunnelling effect in our 
setting. Namely, we consider cut-off potentials and we estimate the rate at which they  suppress the semigroup kernel $p_t(x,y)$ when
$x,y$ are separated by the boundary $K=\partial D$, see Section \ref{sec5} below. 

 In order to deal with irregular domains, we  develop {{} a new} approach, different from one developed by Wu.  For the case of the separation problem, it is still  an elementary and simple
probabilistic argument based on {{} the} Paley-Zygmund  inequality and Blumenthal's zero-one  law.  It becomes more involved Brownian paths analysis for the quantitative description of the tunnelling. 

The assumptions we impose on the potential
$V$,  are optimal within the classes we consider.  The estimates which we {{} discuss in our note}  are strictly connected
to the boundary behaviour of the Brownian motion. We mention \cite{Bou_1}, \cite{Bou} 
as  papers studying diffusion in this direction.  Our motivation for the techniques we use partially comes from 
 the analysis in \cite{Bou_1, Bou}  and \cite{UZ}.

We would like to add that the questions concerning separation can be posed for any semigroup of operators, even
without direct relations to Schr\"odinger mechanics. We mention work \cite{ERSZ} 
where the authors study similar phenomena for certain types of divergence form elliptic operators.
The separation phenomenon for semigroups  is also related to regularity theory of the 
solutions of Partial Differential Equations which  was investigated   in \cite{FKS} and \cite{Tru2}. Interestingly in \cite{FKS} and \cite{Tru2} sufficient and often necessary conditions for 
the regularity of the system (which contradicts the separation) are  formulated in terms of integrability of the coefficients of the corresponding operators whereas we consider an assumption which can be essentially formulated in term of integrability of the potential~$V$.

\section{Preliminaries}\label{sec2}
	We adopt some notation from \cite[Chapter 1]{ito}. We will denote by $\Omega$ the space of all
continuous functions $X:\R_+^0\rightarrow \R^d$. 
We will denote by $X_t$
 its value at $t\in \R_+^0=[0,\infty)$.

Let  $\F$ be
the smallest $\sigma \text{-algebra}$ containing all
cylinders 
$$C_{t_1, t_2,\ldots,t_n, A_1,\ldots , A_1}=\{X\in \Omega : X_{t_1}\in A_1,\ldots ,  X_{t_n}\in A_n, 0\le t_1<t_2<\ldots<t_n\}$$ 
where $A_1,\ldots,A_n$
are Borel sets on $\R^d$ and $n$ is any integer. We define $\F_t$ as the  smallest $\sigma \text{-subalgebra}$ of $\F$
containing all cylinders $C_{t_1, t_2,\ldots,t_n, A_1,\ldots , A_1}$ with $t_n\le t$
 and we set
 \begin{equation}\label{fplus}
 \F_{t^+}=\bigcap_{s>t}\F_s\,.
 \end{equation}

An $\F\text{-measurable}$ function  $\tau:\Omega\rightarrow \R^+\cup\{\infty\}$ will be called a Markov time if $\{m<\tau\}\in \F_t$
for all $0\le t\le \infty$. For a Markov time $\tau$ we define $\F_{\tau^+}$ as the  $\sigma\text{-algebra}$ of all $A\in \F$
such that $A\cap \{m<\tau\}\in \F_t$. 
For given sample paths $X$ we define translation $\Theta_tX=Y$ putting $Y_s=X_{t+s}$. For Markov time $\tau$
we put $\Theta_\tau X=Y $ where $Y_s=X_{s+\tau(X)}$. 

Let  $\Omega_x=\{X\in \Omega: X(0)=x\}$. Denote by $P_x(dX)$ the  Wiener measure on $\Omega_x$  with the density transition corresponding to $\Delta$, see for example \cite[Section 1.4]{ito}. Note that we do not utilize  common probabilistic convention to use the operator $\Delta/2$ as the semigroup generator. Here, we rather use the standard Laplacian because we believe it is a more natural and convenient  definition in the context of Schr\"odinger semigroups studied in our note.
We will call $(\Omega_x, \F, P_x)$ the Brownian motion starting at $x$. 
For  basic properties of the Brownian motion we refer readers to \cite{BG,  ito, KS}.
Next, for  $\sigma\text{-subalgebra }\M \text{ of }\F$ we denote by  $\operatorname{E}_x(f|\M)$ the conditional expectation of $f$ on the probability space
$(\Omega_x, \F, P_x)$.


We will use the following version of classical strong Markov property of the Brownian motion $(\Omega_x, \F, P_x)$:
For $A\in \F_{\tau^+}$ and $B\in \F$ we have
\begin{equation}\label{SMP}
P_x\{X:X\in A\wedge\Theta_\tau X\in  B\wedge X\in \{\tau<\infty\}\}=\int_A P_{X_\tau}(B)P_x(dX)
\end{equation}
We can extend  the above formula to any cylinder 
$$B=C_{\eta_1(X), \eta_1(X)+\eta_2(X),\ldots, \eta_1(X)+\ldots+\eta_n(X) , A_1,\ldots A_n}$$
where $\eta_1,\ldots, \eta_n$ are nonegative $\F_{\tau^+}$ measurable functions. 
This extended form can be easily obtained from \eqref{SMP} for simple functions $\eta_1,\cdots,\eta_n$
and general case by passing to the limit. We refer the reader to \cite[p. 23, 5b)]{ito} see also \cite[p. 169, (8.69)]{Went}.

In the sequel we will use the following observation.
Consider any set $A\in \F_{\tau^+}$.  For a bounded continuous function $ f \colon 
[0,\infty) \times \R^n \to \R$ the expression
$$
\int_A f(\tau, X_\tau )P_x(dX)
$$
defines positive linear functional. So for certain nonegative finite Borel measure $d\mu_A(\tau, w)$ we have
\begin{equation}\label{tau}
\int_A f(\tau, X_\tau)P_x(dX)=\int  f(\tau, w)d\mu_A(\tau, w)
\end{equation}
In what follow it is always clear which set $A$ we consider so we skip the subscript $A$ in $\mu_A$.

Let  ${D}\subset \R^d$ be open, ${D}^c=\R^d\setminus {D}$ be its complement.
Let $V\ge 0$ be locally bounded on ${D}\cup \mbox{int}{({D}^c)}\subset \R^d$.
Following  \cite{CZ}, we define the Feynman-Kac functional by the formula
\begin{equation}\label{eV}
e_V(t)=\exp(-A_V(t))
\end{equation}
 where 
$$
A_V(t)=\int_0^tV(X_s)ds.
$$
Then the one parameter family of operators $\{T_t, t\ge 0\}$ 
 \begin{equation}\label{FK}
T_tf(x)=\operatorname{E}_{x}\{e_V(t)f(X_t)\},
\end{equation}
where by $\operatorname{E}_{x}$ we denote expected value over the Brownian motion starting at point $x\in \R^d$, 
is called Feynman-Kac semigroup, see \cite[(26) p. 76]{CZ}.    It is well known that 
the operators $\{T_t, t\ge 0\}$ form the one parameter strongly continuous symmetric semigroup
of contractions on $L^p(\R^d)$ for all  $1\le p <\infty$. Moreover
$C_c^\infty ({D}\cup  \mbox{int}{({D}^c)})$
is contained in the domain of its infinitesimal generator $-L_V$, and  for all functions     
$\phi \in C_c^\infty ({D}\cup  \mbox{int}{({D}^c)})$ we have
$$L_V\phi(x)=(\Delta -V(x))\phi(x).$$

We say that 
the semigroup  $\{T_t, t\ge 0\}$ separates the sets ${D}$ and ${D^c}$  if all operators $T_t$ preserve the subspace $L^2({D})\subset L^2(\R^d)$,  that is 
 \begin{equation}\label{omega}
T_t(L^2({D}))\subset L^2({D}).
\end{equation}
Note that the operators $\{T_t, t\ge 0\}$ are symmetric  
so  \eqref{omega} implies that 
the subspace $L^2({D}^c)$ is also preserved. 

\medskip

In a part of our discussion in what follows it is convenient to use 
a different  version of Feynmann-Kac formula
based on the notion of the Brownian bridge. Let $\{Y_s, t \ge s\ge 0 \}$ be the Brownian bridge stochastic process connecting points $x,y \in \R^d$ for a definition, see for example  \cite[Example 3, p. 243]{Pro}.  
Using the above notion  the heat kernel corresponding to the Feynmann-Kac semigroup can be written as 
\begin{eqnarray*}
p^{V}_t(x,y)=\int \exp\left(-\int_0^t V(Y_s) ds\right) d\mu_{x,y}(Y{}),
\end{eqnarray*}
where $d\mu_{x,y}$ is  the Brownian bridge  measure defined on the set { $\Omega_{x,y}^t$} of continuous sample paths connecting $x$ and $y$,  normalised in 
such way  that $\int d\mu_{x,y}(\omega)=p_t(x,y)$, see  \cite[Theorem 6.6]{Sim}.
Note that  in our notation $\frac{\mu_{x,y}(\Omega_{x,y}^t)}{p_t(x,y)}$ is the Brownian bridge
probability measure on the set {{} $\Omega_{x,y}^t$}. We will call $d\mu_{x,y}$ the Brownian bridge measure.

\medskip

In our approach the Paley-Zygmund  inequality plays a crucial role. It 
bounds the probability that a positive random variable is small, in terms of its mean and variance.
Let us recall the statement of this result.  

\begin{proposition} 
Suppose that  $Z \ge 0$ is a positive random variable with finite variance and that $0< \theta <1$. 
 Then
$$
    P \big( Z \geq \theta\, \operatorname{E}(Z) \big) \geq (1-\theta)^2\, \frac{\operatorname{E}(Z)^2}{\operatorname{E}(Z^2)}.
    $$
\end{proposition}
\begin{proof} {{} Note that}
$$
    \operatorname{E} (Z) = \operatorname{E} (Z \, \chi_{Z < \theta \operatorname{E}(Z)} ) + \operatorname{E} (Z \, \chi_{Z \geq \theta \operatorname{E}(Z)} ).
$$
Obviously, the first addend is at most $\theta \operatorname{E}(Z)$. By 
the Cauchy-Schwarz inequality the second one is at most
$$
     \operatorname{E} (Z^2)^{1/2}  \operatorname{E}( \chi_{Z \geq \theta \operatorname{E} (Z)})^{1/2} =
     \operatorname{E} (Z^2)^{1/2} P \big(Z \geq \theta\, \operatorname{E}(Z) \big)^{1/2}. 
$$
This proves the required estimate. 
\end{proof}

\section{ Quantum separation for  Feynman-Kac semigroups }
Consider  a closed subset $K \subset \R^d$ with dimension $d\ge 2$. We will impose   fractal like type regularity requirements for $K$. Namely 
for any $\gamma>0$ we define the $\gamma$ neighbourhood   $K_\gamma$ of $K$ by the formula
\begin{equation}\label{K}
K_\gamma=\{x\in \R^d \colon \inf_{y\in K}|x-y| \le \gamma \}.
\end{equation}
In what follows we will always assume that there exist an exponent $0<\alpha<d$ 
and  a positive constant $C_1$ such that 
\begin{equation}\label{vol1}
|K_\gamma\cap B(x,r) | \le C_1  r^\alpha \gamma^{d-\alpha}
\end{equation}
for all $x\in \R^d$ and $1 \ge r \ge \gamma >0$. We also assume that
there exists a positive constant $C_2$ such that 
\begin{equation}\label{vol2}
|K_\gamma\cap B(x,r) | \ge C_2  r^\alpha \gamma^{d-\alpha} 
\end{equation}
for all $x \in K_{r/2}$ and all  $ 1 \ge r \ge \gamma >0$.

The above regularity conditions are frequently considered in the literature and  motivated by the  notion of Minkowski dimension (which is also called box dimension), see 
for example \S 3.1 and Proposition 3.2 of \cite{Fal}.
These conditions  are {{} also} closely  related to the notion of Ahlfors regularity, which is
often used in the context of analysis on metric spaces, see for example \cite{DaS} and  \cite{Hei}.
Using the standard techniques, one can check that these conditions are satisfied for most of the standard fractal constructions including the classical van Koch snowflake curve, 
 again see for example \cite{Fal}.
If the boundary of the region ${D}$ is regular enough, for example 
if  $K=\partial D$ is an immersed  $C^1$  manifold  then {{} it is immediate that} the estimates \eqref{vol1} and \eqref{vol2} hold with 
$\alpha$ equals to topological dimension of $K$. 

We define the distance from $K$ by the formula $d_K(x)=\inf\{d(x,y)\colon y\in K\}$
and then we set 
\begin{equation}\label{beta}
V_\beta= C_Vd_K^{-\beta} .
\end{equation}
The precise value of the constant is irrelevant for our analysis, so in what follows we fix $C_V=1$.

 Before we state our first  result, Theorem~\ref{multiD} below,  we would like
	make the following observation which  explains  why we consider only the range of exponents $\alpha > d-2$ in what follows. 

\begin{remark}
	It is not difficult to show that if {$X$} is the Brownian motion in the Euclidean space $\R^d$  starting at the origin then 
	$$
	P_0\left(\int_0^\delta |X_s|^{-2}  ds = \infty, \, \forall\delta>0\right)=1.
	$$
	Hence there is no point to study the case $\beta \ge 2$ and  we can assume that $\alpha > d-2$ in the following statement.
\end{remark}


\begin{theorem}\label{multiD}
Assume that a closed subset $K \subset \R^d$, for some $d\ge 2$ satisfies conditions \eqref{vol1} and \eqref{vol2}
with some $d>\alpha > d-2$. Suppose next that {$X$} is the Brownian motion starting at point $x$  contained in  $K$ that is  such that $X_0=x\in K$. Then 
$$
P_x\left(\int_0^\delta V_\beta(X_s)ds = \infty, \, \forall\delta>0\right)=1
$$
for every $\beta$ such that  $\beta+\alpha \ge d$. 
\end{theorem}

\begin{proof} Note that by taking intersection of $K$ with a closed ball $\overline {B(x,1)}$
we can assume without loss of generality that the set $K$ is compact. 
Next note that it follows from \eqref{vol1} and \eqref{vol2}
 that if one takes  $1> a >0$ such that $C_2/2 \ge a^{d-\alpha}C_1$ then for $C_3=C_2/2$ 
\begin{equation}\label{vol3}
|(K_\gamma \setminus K_{a \gamma})\cap B(x,r) | \ge C_3 r^\alpha \gamma^{d-\alpha} \quad \forall_{x \in K_{r/2}}\,.
\end{equation}
Now for any $n\in \N$ we set 
$$
K'_n= K_{a^n} \setminus K_{a^{n+1}},
$$
where $a$ is the constant from estimate \eqref{vol3}. 
 We define a sequence of random variables~$Z_n$ by  the following formula. 
$$
Z_n(X)=\int_0^{\delta^2}\chi_{K'_n}(X_s)ds
$$
where $\chi_{K'_n}$ is the characteristic function of the set $K'_n$ described above
and {$X$} is the Brownian motion process starting at some fixed point $x\in K$.

Following the idea of \cite{UZ},  
we shall verify assumptions of the Paley-Zygmund inequality for each random variable~$Z_n$.  To that end take some $\delta>0$ and 
set  $$b_n = a^{n(d-\alpha)}\delta^{2-d+\alpha}.$$ We shall prove the following estimates for the expected values of  $Z_n$ and $Z_n^2$
\begin{eqnarray}\label{aaa}
\operatorname{E}_x(Z_n)  \ge C   b_n   
\end{eqnarray}
and 
\begin{eqnarray}\label{bbb}
\operatorname{E}_x(Z_n^2)  \le c  b_n^2
\end{eqnarray}
valid for all $n\in \N$ such that $a^n \le \delta^2$. The constants $c,C$ in \eqref{aaa} and \eqref{bbb} do 
not depend on $n$. 
In order to prove \eqref{aaa} we note that if $d>2$ then for positive constants $C, C', c',c >0$  
   \begin{equation}\label{e8}     
C' r^{2-d}\exp\left(-\frac{c'r^2}{\delta^2}\right)  \le \int_0^{\delta^2}{t^{-d/2}}\exp\left(\frac{-r^2}{4t}\right)dt     \le  Cr^{2-d}\exp\left(-\frac{cr^2}{\delta^2}\right).     \end{equation}
Whereas for $d=2$ 
   \begin{eqnarray}\label{2d}     
C'(1+|\log( \delta /r)|) \exp\left(-\frac{c'r^2}{\delta^2}\right)  \le \int_0^{\delta^2}{t^{-d/2}}\exp\left(\frac{-r^2}{4t}\right)dt    \\ \le  C(1+|\log(\delta /r)|)\exp\left(-\frac{cr^2}{\delta^2}\right). 
\nonumber    \end{eqnarray}

We first discuss the case  $d>2$.  Since $p_t(x,y)=  (4\pi t)^{-d/2}e^{-\frac{|x-y|^2}{4t}}$ is the transition density of the considered process $X$, using Fubini's theorem, we 
have  
\begin{eqnarray*}
\operatorname{E}_x(Z_n)&=&  \int \int_0^{\delta^2} \chi_{K'_n}(X_t)  dtP_x(dX)\\   &=&
\int_0^{\delta^2}\int_{\R^d} \chi_{K'_n}(y) {(4t\pi)^{-d/2}}\exp\left({-\frac{|x-y|^2}{4t}}\right) dy dt\\
&=&\int_{\R^d} \int_0^{\delta^2} \chi_{K'_n}(y) {(4t\pi)^{-d/2}}\exp\left({-\frac{|x-y|^2}{4t}}\right) dt dy,
\end{eqnarray*}
By \eqref{e8} and then by \eqref{vol3}
 \begin{eqnarray*}
\operatorname{E}_x(Z_n)&=&\int_{\R^d} \int_0^{\delta^2} \chi_{K'_n}(y) {(4t\pi)^{-d/2}}\exp\left({-\frac{|x-y|^2}{4t}}\right) dt dy
\\ &\ge& C  \int_{|x-y| \le \delta} \chi_{K'_n}(y) |x-y|^{2-d} dy
\ge c \delta^{(2-d)}
\int_{K'_n\cap B(x,\delta)} dy
\\ &\ge& c C_3  \delta^{(2-d)}a^{n(d-\alpha)} \delta^\alpha
= c' b_n.
\end{eqnarray*}  
This proves estimate \eqref{aaa}.

\medskip

In the next step of the proof  we will verify estimate \eqref{bbb}. 
\begin{eqnarray*}
Z_n^2(X)&=&\int_0^{\delta^2}\int_0^{\delta^2} \chi_{K'_n}(X_s) \chi_{K'_n}(X_t)ds  dt\\
&=& \int_{s\le t  \le \delta^2} \chi_{K'_n}(X_s) \chi_{K'_n}(X_t)ds  dt
+\int_{t\le s  \le \delta^2} \chi_{K'_n}(X_s) \chi_{K'_n}(X_t)ds  dt.
\end{eqnarray*}
Now for $t>s$ it follows from the independence of $X_s$ and $X_t-X_s$ 
\begin{eqnarray*}
&&\int \int_{s\le t  \le \delta^2} \chi_{K'_n}(X_s) \chi_{K'_n}(X_t)ds  dtP_x(dX)\\ &=&\int\int_{s\le t  \le \delta^2}
\chi_{K'_n} (X_s) \chi_{K'_n}(X_t-X_s+X_s)
 dsdt P_x(dX)
 \\&=&
\int_0^{\delta^2} \int_0^{t}\int\int \frac{\chi_{K'_n}(y)\chi_{K'_n}(z+y)}{(4\pi)^d (s(t-s))^{d/2}}\exp\left({-\frac{|x-y|^2}{4s}}\right)\exp\left({-\frac{|z|^2}{4(t-s)}}\right)  dz dy dsdt.
 \end{eqnarray*}
 Next, by \eqref{e8}
 \begin{eqnarray}\nonumber
\int_0^{\delta^2} \int_0^{t}\int\int \frac{\chi_{K'_n}(y)\chi_{K'_n}(z+y)}{(4\pi)^d (s(t-s))^{d/2}}\exp\left({-\frac{|x-y|^2}{4s}}\right)\exp\left({-\frac{|z|^2}{4(t-s)}}\right)  dz dy dsdt \\  \nonumber \le 
C \int\int\frac{\chi_{K'_n}(y)\chi_{K'_n}(z+y)}  {|x-y|^{d-2}|z|^{d-2}}\exp\left(-\frac{c|x-y|^2}{\delta^2}\right)\exp\left(-\frac{c|z|^2}{\delta^2}\right) dz dy=:I.
 \end{eqnarray}
To estimate the term $I$ we recall from  the beginning of  proof that without loss of generality we can assume that $K \subset  {B(x,1)}$. For any $m \in \N$ set 
$$A_m=\{ z \in \R^d \colon \, a^{m+1} < |z|  \le a^{m}\}.$$ Then by \eqref{vol1}
 \begin{eqnarray}\label{l14}
&& \int\chi_{K'_n}(z+y)  \nonumber
 |z|^{2-d}\exp\left(-\frac{c|z|^2}{\delta^2}\right) dz \\ 
& =&\sum_{m\in \N} \int_{A_m}\chi_{K'_n}(z+y)  \nonumber
 |z|^{2-d}\exp\left(-\frac{c|z|^2}{\delta^2}\right) dz\\
& \le &
 C\sum_{a^m < \delta} a^{m(2-d)} a^{m\alpha}a^{n(d-\alpha)}
 + C\sum_{a^m \ge \delta} a^{m(2-d)} a^{m\alpha}a^{n(d-\alpha)}\exp\left(-\frac{ca^{2m}}{\delta^2}\right)
 \\
& \le &
 Ca^{n(d-\alpha)}\delta^{2-d+\alpha}+
 Ca^{n(d-\alpha)}\delta^{2-d+\alpha} \sum_{a^m \ge \delta}\left(\frac{a^m}{\delta}\right)^{2-d+\alpha}\exp\left(-\frac{ca^{2m}}{\delta^2}\right)
 \nonumber
 \\
 &\le &
  Ca^{n(d-\alpha)}\delta^{2-d+\alpha}=Cb_n.
  \nonumber
\end{eqnarray}
By the above estimate 
\begin{eqnarray*}
 \int\int\chi_{K'_n}(y)\chi_{K'_n}(z+y)  |x-y|^{2-d}\exp\left(-c\frac{|x-y|^2}{\delta^2}\right)
 |z|^{2-d}\exp\left(-c\frac{|z|^2}{\delta^2}\right) dz dy\\ \le 
 Ca^{n(d-\alpha)}\delta^{2-d+\alpha}\int\chi_{K'_n}(y) |x-y|^{2-d}\exp\left(-c\frac{|x-y|^2}{\delta^2}\right) dy. 
\end{eqnarray*}
Now the repetition of the calculation  in \eqref{l14} applied to the last integral above, that is 
$ \int\chi_{K'_n}(y) |x-y|^{2-d}\exp\left(-c\frac{|x-y|^2}{\delta^2}\right) dy  $ yields  the required estimate 
$$I \le C b_n^2=C\left( a^{n(d-\alpha)}\delta^{2-d+\alpha}\right)^2.$$
This proves estimate \eqref{bbb}.

\medskip
 
Next,  by the Paley-Zygmund  inequality with $\theta =1/2$
it follows from  estimates \eqref{aaa} and \eqref{bbb} that
 there exists a constant $\sigma>0$ independent of $n$ and $\delta$  such that for an appropriate constant $c$
\begin{eqnarray}\label{PZest}
P_x\Big(Z_n \ge  c a^{n(d-\alpha)}\delta^{2-d+\alpha}=cb_n  \Big) \ge P_x\left( Z_n \ge \frac{\operatorname{E_x}(Z_n)}{2}\right) \ge\frac{\operatorname{E_x}(Z)^2}{4\operatorname{E_x}(Z^2)} \ge \sigma  .
\end{eqnarray}
Hence for any $\delta>0$.
\begin{eqnarray}\nonumber 
P_x\left( \int_0^{\delta^2}\chi_{K'_n}(X_s)ds \ge   c a^{n(d-\alpha)}\delta^{2-d+\alpha} \quad 
\mbox{for infinitely many $n$}  \right) \\=
	P_x\left( \bigcap_{N=1}^\infty \bigcup_{n=N}^\infty\left\{\int_0^{\delta^2}\chi_{K'_n}(X_s)ds \ge   c a^{n(d-\alpha)}\delta^{2-d+\alpha} \right\}  \right) 
	\ge \sigma.\label{12}
\end{eqnarray}

Now  consider the additive  functional $A_{V_\beta}$ with 
$V_\beta= d_K^{-\beta}$ for some $\beta \ge d-\alpha$. Then, by \eqref{PZest} and  \eqref{12}, for any sequence $\delta_j$ decreasing 
to $0$ we have
\begin{eqnarray*}
 &&\hspace{-1cm}P_x\left(A_{V_\beta}(\delta^2)=\int_0^{\delta^2}V_\beta(X_s)ds = \infty, \quad \forall \delta>0\right)
\\&=&P_x\left( \bigcap_j \left\{\int_0^{\delta_j^2}V_\beta(X_s)ds = \infty \right\}\right)
=\lim_{j\to \infty}P_x\left( \int_0^{\delta_j^2}V_\beta(X_s)ds = \infty \right) \\
&\ge& \lim_{j\to \infty}P_x\left( \int_0^{\delta_j^2}\chi_{K'_n}(X_s)ds \ge   c a^{n(d-\alpha)}\delta_j^{2-d+\alpha} \quad \mbox{for infinitely many $n$}  \right) \ge 
\sigma.
\end{eqnarray*}
  Before we continue our discussion we shall justify the first of the above inequalities in more detailed way. Assume we are given a sample path   $X \in \Omega_x$ with 
$$
a^{n(\alpha-d)}\int_0^{\delta_j^2}\chi_{K'_n}(X_s)ds \ge c \delta_j^{2-d+\alpha}
$$
for all $n \in A_X $ and some infinite subset $ A_X \subset \N$.
Then since $a^{n(\alpha-d)}{V_\beta}(x) \ge 1$ for all $x\in K'_n$
it follows that 
$$
\int_0^{\delta_j^2}V_\beta(X_s)ds=
\sum_n \int_0^{\delta_j^2}V_\beta(X_s)\chi_{K'_n}(X_s)ds
\ge c\sum_{n \in A_X} \delta_j^2= \infty. 
$$
Hence the first inequality follows.

Now the event $$\Omega_{V_\beta}=\left\{X: \int_0^{\delta^2}V_\beta(X_s)ds = \infty \colon \quad \forall \delta>0\right\}$$ is measurable with respect 
to $\sigma$-field  ${\F_{0^+}}$ defined by \eqref{fplus} so by Blumenthal's zero-one law,   see for example \cite{Blum} or \cite[p. 25, Problem 2]{ito},  it must have  probability equal to $0$ or $1$. 
Thus $P(\Omega_{V_\beta})=1$. This ends the proof of Theorem~\ref{multiD} in the case $d>2$. For $d=2$ the proof is a simple modification of the above argument. One can essentially repeat the same calculation just replacing estimates \eqref{e8} by estimates \eqref{2d} corresponding to the case $d=2$. We skip the details here. 
\end{proof}

We are now in position to state the main result of this section. 
\begin{theorem}\label{sep}
Suppose that the set ${D} \subset \R^d$ is 
open simply connected  and that its boundary  $K=\partial {D}$ satisfies conditions \eqref{vol1} and \eqref{vol2} for some 
$d>\alpha >0$ and $d\ge 2$.
Let  $\{T_t, t\ge 0\}$ be the Feynman-Kac semigroup generated by  $L_{V_\beta}=\Delta -V_\beta$, where the potential 
$V_\beta$ is defined by \eqref{beta}.  Assume also that  $\alpha+\beta \ge d$.

Then the subspace $L^2({D})$ of $L^2(\R^d)$ is  invariant under the action $T_t$
that is 
$$
T_t(L^2({D}))\subset L^2({D})
$$
for all $t\ge 0$. 
\end{theorem}

\begin{proof} 
Recall that  $\Omega_x$ is the set of paths starting from $x\in {D}$. Let
 $\tau(X)$ denotes the  first hitting time of $K=\partial D$
for $X$.
Consider the  set $$\Omega_V=\left\{X: \int_0^{\delta^2}V(X_s)ds = \infty \colon \quad \forall \delta>0\right\}$$

 By Theorem \ref{multiD} and  the strong Markov property \eqref{SMP}, with  $A=\{\tau \le t\}\in \F_{\tau^+} $ and $B=\Omega_V$, see also \cite[Point 5b), page 23]{ito}
\begin{eqnarray*}
P_x\left(X\in \Omega_x: \tau\le t, \Theta_\tau X\in \Omega_V \right)
=\int_{\{\tau\le t\}} P_{X_\tau}(\Omega_V)P_x(dX)\\
=P_x(\{\tau\le t\}). 
\end{eqnarray*}
 Hence 
 the event: $$\mbox{ for all $\delta>0$ we have }\int_\tau^{\tau+\delta^2}V(X_s)ds = \infty$$ 
 holds a.s on $\{\tau\le t\}$.  Clearly for any path $X$ if $X_0 \in D$ 
 and $X_t \in D^c$ then  $\tau(X) < t$ so now
the theorem follows by applying Feynman-Kac formula \eqref{FK}.
\end{proof}

\section{Singularity of $V$ forcing separation.}\label{bridgee}
In this section we prove that Theorem~\ref{sep} is optimal, that is that the condition 
$\alpha +\beta \ge d$ is also necessary. By $p_t(x,y)$ we denote the Gaussian distribution (corresponding to $\Delta$ rather than $\Delta/2$, see the discussion in Section \ref{sec2}). 
Recall that  
$$p_t(x,y)=  (4\pi t)^{-d/2}e^{-\frac{|x-y|^2}{4t}}.$$ 
Then as before by $p_t^V(x,y)$ we denote the kernel 
of the  Feynman-Kac semigroup corresponding to the positive potential $V$
defined by \eqref{FK}. 
In these terms  Theorem~\ref{sep} can be stated in the following way: 
For any $\beta$ such that $\alpha +\beta \ge d$
$$
p_t^{V_\beta} (x,y) = 0  \quad \forall {t>0, \quad  x\in D} \quad \mbox{and} \quad \forall {y\in D^c}.
$$
 Note that the kernel $p_t^{V_\beta}$ is symmetric so it automatically follows that 
if this is the case then $p_t^{V_\beta} (x,y) = 0$ also  whenever ${x\in D^c}$
and $y\in D$. 
	
Next  we shall show that if $\alpha +\beta < d$ then 
$$
p_t^{V_\beta} (x,y) > 0 \quad \forall {t>0, x,y \in \R^d},
$$
see Theorem \ref{lowerbdth} and Corollary~\ref{c6} below. 

First for any $t>0$ we set 
$$
\Gamma_t(x,y) = \int_0^{t/2} p_s(x,y) ds= \int_0^{t/2} (4\pi s)^{-d/2}e^{-\frac{|x-y|^2}{4s}}ds.
$$
\begin{theorem}\label{lowerbdth}
Assume that $V\in L_{loc}^1(\R^d)$ is a  positive locally integrable potential.
Suppose also that for some fixed $t$ 
$$
V*\Gamma_t(x) + V*\Gamma_t(y) < \infty. 
$$
Then  $p^{V}_t(x,y) >0$. 
\end{theorem}
\begin{proof}

Recall that 
 $d\mu_{x,y}$ is  the Brownian bridge  measure defined on the set { $\Omega_{x,y}^t$} of continuous sample paths connecting $x$ and $y$ and normalised in 
such way  that $\int d\mu_{x,y}(\omega)=p_t(x,y)$, see Section \ref{sec2} above.

Now the transition density formula  of the Brownian bridge, see for example 
\cite[p. 359 (6.28)]{KS}
and  Chebyshev's inequality yield 
\begin{eqnarray*}
&&  \hspace{-1.5cm}\mu_{x,y} \left(\left\{ \int_0^t V(Y_s) ds \ge A    \right\} \right)
\le \frac{1}{A}\int \int_0^t V(Y_s) ds d\mu_{x,y}(Y{})\\
&=& \frac{1}{A} \int_0^t \int_{\R^d}V(z) p_s(x-z)p_{t-s}(z-y) dz ds \\ &\le&
\frac{C}{At^{d/2}}\left(\int_0^{t/2} \int_{\R^d} V(z) p_s(x-z)dz ds  + \int_0^{t/2} \int_{\R^d}V(z) p_s(y-z) {{} dz ds}  \right) \\&& \hspace{5cm}
\le \frac{C}{At^{d/2}} \left(V*\Gamma_t(x)+V*\Gamma_t({{}y})\right).
\end{eqnarray*}
Hence for sufficiently large $A$ 
$$
\mu_{x,y}\left( \left\{ \int_0^t V(Y_s) ds \le A    \right\}\right) \ge p_t(x,y) - \frac{C}{At^{d/2}} \left(V*\Gamma_t(x)+V*\Gamma_t(y)\right) \ge \frac{1}{2} p_t(x,y).
$$
Thus 
$$
p^{V}_t(x,y) = \int \exp\left(-\int_0^t V(Y_s) ds\right) d\mu_{x,y}(Y{})\ge \frac{1}{2}e^{-A} p_t(x,y) >0.
$$
{{} This concludes proof of Theorem \ref{lowerbdth}.}
\end{proof}
As a direct consequence of {Theorem \ref{lowerbdth} we obtain the following 
corollary}
\begin{coro}\label{c6}
Under the assumptions of {\rm Theorem~\ref{sep}}
the condition  $\alpha+\beta\ge d$ is necessary for separation.
\end{coro}
{{}
\begin{proof}
The proof is essentially repetition of the proof of estimate \eqref{aaa} from the proof of 
Theorem~\ref{multiD}. More precisely one can notice that if we set $t=\delta^2$ then
$$
V*\Gamma_t(x) \le C\sum_{n>0} a^{-n\beta} E(Z_n)  \le C\sum_{n>0} a^{-n\beta} a^{n(d-\alpha)}.
$$
Recall that $a<1$ so the above sum is finite if $\alpha + \beta < d$. Now Corollary~\ref{c6} follows 
from Theorem~\ref{lowerbdth}.
\end{proof}
}
{{} As an illustration of Theorem~\ref{sep} and Corollary \ref{c6} we would like to describe the construction of  the van Koch snowflake
curve. }

\begin{example} 
Van Koch snowflake. Consider an equilateral triangle $K_0$
with sides of unit length. Next define a curve $K_1$ by  replacing the middle 
of all edges of $K_0$ by the two sides of the equilateral triangle based
on the middle every segment. 
Next define  $K_2$ by repeating the same  procedure on each of the twelfth edges  of $K_1$.
Von Koch snowflake, which we denote by $K$ is the self-similar set obtained by iteration of this 
procedure. Its Minkowski dimension is equal to $\alpha =\log4/\log3$ {{} and it satisfies assumptions \eqref{vol1} and \eqref{vol2}  with this $\alpha$ and $d=2$.} 
  
\end{example}
{{}Thus as a straightforward consequence of our results we obtain the following corollary 
\begin{coro}\label{c7}
Suppose that $D \subset \R^2$ is the region of the plane inside the Von Koch snowflake $K$ and $p_t^{V_\beta}$ is the heat kernel 
corresponding to the operator $\Delta -d_K^{-\beta}$. Then for any $\beta \ge 2- \log4/\log3$
$$
p_t^{V_\beta} (x,y) = 0 \quad \forall {t>0, x\in D} \quad \mbox{and} \quad \forall {y\in D^c}.
$$
When $\beta < 2- \log4/\log3$ then 
$$p_t^{V_\beta} (x,y)>0   \quad \forall  x,y\in \R^2   \quad \mbox{and} \quad \forall t>0.$$ 
\end{coro}}

\section{Estimates of the rate of separation for truncated potentials}\label{sec5}

In this section we consider the potential
$$
V_\beta^A(x)= C_VA^{\beta}\chi_{\{d_K(x){{}\le } A^{-1}\}}={{} C_VA^{\beta}\chi_{K_{A^{-1}}}}(x)
$$
We fix a supercritical exponent $\alpha+\beta>d$.
Denote by $p_t^A(x,y)$  the  kernel of the Feynman-Kac semigroup 
 generated by {{} the operator} $$-L_{V^A_\beta}=\Delta-V_\beta^A$$ with the natural domain.
It is well known that
the functions  $p_t^A(x,y)$ are continuous in $x,y,t\in \R^d\times  (0,\infty)$.

It is convenient for us to introduce at this point the following definition of the uniform domain type.  Our definition is a variant of 
Definition~3.2 of \cite{GS} but is motivated by the definition of  NTA (nontangentially accessible) domains introduced
by Jerison and Kenig in \cite{JK}. See also the discussion in Section 3.1.3 of \cite{GS}.

\begin{definition}
Let ${D}\subset \R^d$ be a connected subset of  $\R^d$. 
We say that  ${D}$ satisfies the   inside NTA condition if there are constants $c, C$ such that, for
any $x, y \in D$ in the interior of $D$ there exists a continuous curve $\gamma_{x,y} \colon [0,1] \to D$ such that $\gamma_{x,y}(0)=x$ 
and $\gamma_{x,y}(1)=y$  and the following two properties are satisfied 
\begin{itemize}
\item[1.] The length $L(\gamma_{x,y})$ is at most $C|x-y|$.
\item[2.] For any $z\in \gamma_{x,y}([0,1])$ 
$$
 d_{\partial D} (z) \ge c \min( L(\gamma_{x,z}), L(\gamma_{y,z})  ).
$$
\end{itemize}
We say that ${D}$ is  an 
$NTA$ domain if both  ${D}$ and ${D}^c$ satisfies the inside $NTA$ condition.
\end{definition}

We say that an open domain  ${D}\subset \R^d$ is separated from 
a given ball  $B_0=B(x,r) \subset \R^d$ if $2B_0=B(x,2r) \subset {D}^{c}=\R^d \setminus {D}$. 
The following  statement is the main result discussed  in this section. 

\begin{theorem}\label{main1}
Suppose that set ${D} \subset \R^d$ is 
open simply connected   and that its boundary  $K=\partial {D}$ satisfies conditions 
\eqref{vol1}  and \eqref{vol2}  for some 
$d> \alpha >0$ and that $\alpha+\beta>d \ge 2
$. Let  ${D}$ be an  $NTA$ domain.
There exists a constant $\sigma>0$ such that for any point  $x\in {D}$, any ball $B_0$ separated from ${D}$, and any positive $t>0$ there exists a constant $C=C(x,t)$ such that 
\begin{equation}\label{mainthesis}
\int_{B_0}p_t^A(x,y) dy \le C A^{-\sigma}
\end{equation}
for all $A>0$. 
\end{theorem}

\begin{remark} By the  standard elliptic estimates (or by a slightly more technical variant of our argument) one can obtain a pointwise estimate $p_t^A(x,y)  \le C A^{-\sigma}$.
\end{remark}
\begin{remark}
 In this work we are interested in the existence of a positive $\sigma$ satisfying \eqref{mainthesis}.
{ We want to point out however that our methods can yield estimates \eqref{mainthesis} for a larger range of parameter $\sigma$ at the cost of some further work and more complex and tedious argument.  
We do not discuss the details of such strengthen result  in the present paper.}
\end{remark}

\begin{remark} The value of $\sigma $  in our approach depends on  the domain. 
The constant $C$ depends on the domain and the value of the multiplicative  constant $C_V$. 
\end{remark}
For the sake of simplicity we will consider only the case  $C_V=1$.
Without loss of generality we can assume that $t=1$. First we prove a series of technical lemmata concerning properties of the Brownian motion.
\begin{lemma}\label{l8}
 There exist constants ${{} 0<\zeta<1}$ and $ \sigma_0>0$ such that
\begin{equation*}
P_x\left( \int_0^{\delta^2}V_\beta^{A}(X_t) dt  < {{} \zeta}  \operatorname{E}\bigg( \int_0^{\delta^2}V_\beta^{A}(X_t) dt  \bigg)\right)
\le 1- \sigma_0 
\end{equation*}
for all $\delta>0$  and $A>0$ and every  starting point $X_0=x\in K$.
\end{lemma}
\begin{proof}
To prove the statement we use an argument similar to one which we used to verify  \eqref{PZest}. We have to replace  $\chi_{K'_n}$ in the definition of $Z_n$ by
$\chi_{\{d_K(x){{}\le } A^{-1}\}}$. Then the statement follows
by replicating the calculations leading to  \eqref{PZest} which we use in the proof of Theorem \ref{multiD}. To avoid repetition we omit the details.
\end{proof}

Let ${\bf w}_\lambda=\lambda \dot (1,1,\ldots, 1)\in \R^d$ where $\lambda \in \R$ will be specified later. Consider the decomposition of $\R^d$ into  congruent cubes of side length $\delta$ obtained by the $\delta \Z^d$ translations of ${\bf w}_\lambda+[0,\delta]^d$ and denote by
$Q_1, \ldots, Q_M$ all these cubes, such that {$Q_j \cap K \neq \emptyset$} for $j=1,\ldots, M$.  
By \eqref{vol1} one  immediately gets  
$
 \frac1C \delta^{-\alpha}\le M \le C\delta^{-\alpha}
$
where $C$ does not depend on $\lambda$.
We fix some small $v>0$, $\delta \approx A^{-v}$ and $H\in \N$ in such a way that $2H+2=\delta^{-2}$.
For any $h\in  \{0, \ldots, H\}$  we set 
$$
I_h=[2h\delta^2, (2h+1)\delta^2) \quad \mbox{and} \quad J_h=[(2h+1)\delta^2, (2h+2)\delta^2).
$$

Let $1\le H_0 \le H$ be a  fixed  number. 
We will consider multi-indexes of the form  $j_1, \ldots, j_{H_0}, k_1, \ldots, k_{H_0}$
such that $1 \le j_s \le M$, $0 \le k_s \le H$, $k_1 < k_2 \ldots < k_{H_0}$ for all $ 1 \le s \le H_0$. 

Let $\Omega$  be the set of all Brownian paths $X$ such that $X_0=x$ and 
$X_1\in B_0$.
 We define the  family of the  subsets 
${{}\Phi}_{j_1, \ldots, j_{H_0}, k_1, \ldots, k_{H_0}}\subset \Omega$ of $\Omega$ by requiring the following conditions: 
\begin{enumerate}
\item
For every $ 1 \le s \le H_0$ there exists time $t'_s \in I_{k_s}$ such that  for the path $X\in \Omega $,  $X_{t'_s}\in K$. By $t_s$ we denote the smallest 
such $t'_s$ (the first hitting time of $K$ for  $t \in I_{k_s}$ )
\item For every $ 1 \le s \le H_0$, it holds that  $X_{t_s}\in Q_{j_s}$.
\item 
For every  $t \in I_h$ and for any $h$ not listed in the sequence 
$k_1, \ldots, k_{H_0}$ one has $X_t \notin K$.
\end{enumerate}
  The sets ${{}\Phi}_{j_1, \ldots, j_{H_0}, k_1, \ldots, k_{H_0}}\subset \Omega$ do not need to be disjoint. We observe that a given 
sample path belongs to the two different sets ${{}\Phi}_{j_1, \ldots, j_{H_0}, k_1, \ldots, k_{H_0}}$	if the first hitting point $X_\tau$ for some $\tau \in I_h$ belongs to two different cubes $Q$ of the grid. It is possible only if $X_\tau$ belongs to common wall of two neighbouring cubes 
$Q$. We have

{\bf Claim. } There exists $\lambda \in [0,\delta]$ such that the  Wiener measure of the set $\lambda $ of trajectories which do not uniquely determine cubes  
$Q_{j_1} \ldots,Q_{j_{H_0}}$ is zero. 
\begin{proof}
Indeed,   denote by $\tau\in I_{k_j}$ the first hitting time of $K$.  
Fix the   $d-1$ or lower dimensional common wall ${\bf W}$ of two cubes and observe that all of its translations
 ${\bf W}_\lambda  ={\bf W+w}_\lambda$ are pairwise disjoint. Hence the events  $\{X_\tau \in {\bf W}_\lambda\}$
are pairwise disjoint, so only for at most countably many of  ${ \lambda}$ will the event   $\{X_\tau \in {\bf W}_\lambda \}$
 have positive Wiener measure.  The claim follows. 
\end{proof}
From now on we fix $\lambda$ given by the claim and consider the grid of cubes corresponding to ${\bf w}_\lambda$.
We subtract the  set $E_\lambda$ from $\Omega$ and ${{}\Phi}_{j_1, \ldots, j_{H_0}, k_1, \ldots, k_{H_0}}$ and
 denote the new sets again  by $\Omega$
and ${{}\Phi}_{j_1, \ldots, j_{H_0}, k_1, \ldots, k_{H_0}}$. 
Now it is straightforward to see that the sets ${{}\Phi}_{j_1, \ldots, j_{H_0}, k_1, \ldots, k_{H_0}}$ are pairwise disjoint.
Moreover the set
$$\bigcup_{j_1, \ldots, j_{H_0}, k_1, \ldots, k_{H_0}} {{}\Phi}_{j_1, \ldots, j_{H_0}, k_1, \ldots, k_{H_0}}$$
 (up to a subset of measure zero)
consists of  such $X \in \Omega$  that  the path
$X $ for all $h=1,\ldots, H $ intersects  $K$ for exactly $H_0$ out of $H$ intervals  $I_h$. 
We will use {{} these} facts in the sequel.

Now  by $q_{j_s}$ we denote the center of $Q_{j_s}$ and set 
$$\gamma_{\delta,H}=\delta ^2(\log{H})^{-3}=(\log{H})^{-3}/(2H+2).
$$ 
Next,  
for any $\eta \ge 1$  and $x,y \in \R^d$, such that 
$$
|x-q_{j_s}|\le \eta \delta \sqrt{\log{H}}, \mbox{ \hskip1cm and \hskip.5cm}  |y-q_{j_s}|\le \eta \delta  \sqrt{\log{H}}
$$
{we define a set $\Psi_0$ by the formula 
$$
{} \quad  \quad{{} \Psi_0}= \left\{X \in \Omega_{x,y}^{2\delta^2} \colon  \tau_{j_s}\le \delta^2 ,\,
|X_{\tau_{j_s}}-X_{\tau_{j_s}+\gamma_{\delta,H} }|  \le \frac{\eta \delta}{ \log{H}} \right\},
$$}
{where   
$\tau_{j_s}=\inf\{t\le \delta^2:X_t\in K \cap Q_{j_s}\}$ or $\tau_{j_s}=\infty$ if the set is empty. 
{{} Then we set  
\begin{eqnarray}\label{p}
 p(x,y)=\mu_{x,y}(\Psi_0),
\end{eqnarray}
where $d\mu_{x,y}$ is the  Brownian bridge measure, see the definition in Section~\ref{sec2}.
Then we define 
$$
 \quad  \quad{{} \Psi_{0,r}}= \left\{X \in \Omega_{x} \colon  \tau_{j_s}\le \delta^2 ,\,
 |X_{\tau_{j_s}}-X_{\tau_{j_s}+\gamma_{\delta,H} }|  \le \frac{\eta \delta}{ \log{H}}, |X_{2\delta^2}-y|\le r \right\}.
$$}
We put  
$$
{{} \Psi_1}= \left\{X \in {{} \Psi_0} \colon 
\int_{\tau_{j_s}}^{\tau_{j_s}+\gamma_{\delta,H} } V_\beta^A(X_u)du  < {{} \zeta} E\left( \int_{\tau_{j_s}}^{\tau_{j_s}+\gamma_{\delta,H} } V_\beta^A(X_u)\right) du  \right\}
$$
where the constant $0<{{} \zeta}<1$ has been defined in Lemma \ref{l8}
and similarly as above we set 
\begin{eqnarray}\label{pt}
\tilde{p}(x,y)=   \mu_{x,y}(\Psi_1). 
\end{eqnarray}
We define the set $\Psi_{1,r}$ replacing $\Psi_0$ by $\Psi_{0,r}$ in the definition of $\Psi_1$. }

In what follows we will need the relation $\tau_{j_s}+\gamma_{\delta,H}\le \tau_{j_{s+1}}$.
We have ensured this property by separating subsequent intervals $I_s,I_{s+1}$ by $J_s$.

The following elementary  observation is critical for our argument
\begin{lemma}\label{l9} 
Under the above definitions  there exists a constant   $ {\xi} <1$   such that 
$$
\tilde{p}(x,y) \le {\xi} p(x,y).
$$
uniformly for all indices $x,y, A, j_s$ defined above and  {{} sufficiently large $H\ge H_{min}(\eta)$.}
\end{lemma}
\begin{proof}
In the proof of the lemma, for notational  convenience, 
we put $\tau=\tau_{j_s}$.
By the definition of ${{} \Psi_0}$, the variable
$\tau$ is finite for  all paths in ${{} \Psi_0}$. Denote by $d\mu (\tau, w) $ the joint distribution of the variables $\tau, X_{\tau}$ defined by \eqref{tau}  for $\Psi_0=A$. 
Obviously $d\mu$ is supported on $S= [0,\delta^2]\times ( K \cap Q_{j_s})$. 
Next for any subset $G \subset \R^d$ put
\begin{eqnarray*}
\nu_1(G)= P_0\left(X_{\gamma_{\delta,H} }\in G\right).
\end{eqnarray*}
{{} Note that at this point $\nu_1$ is just the Gaussian distribution and recall that $\gamma_{\delta,H}=\delta ^2(\log{H})^{-3}$.
Next, we consider the  set $K_r$ defined by \eqref{K} and  we observe that {{} by the strong Markov property \eqref{SMP} 
 applied to $A=\Psi_{0,r}$, $\eta_1(\tau)=\tau+\gamma_{\delta,H}$,
$\eta_2(\tau)=2\delta^2-\gamma_{\delta,H}-\tau$ and $B=\left\{|w-X_{\tau_{}+\gamma_{\delta,H} }| \le  \frac{\eta \delta}{\log{H}}, |X_{2\delta^2}-y|\le r  \right\}$, it follows that}
\begin{eqnarray*}p(x,y)&=&\mu_{x,y}({{} \Psi_0})=\lim_{r\rightarrow 0}|K_r|^{-1}P_x(\Psi_{0,r})=\\ &=&
\lim_{r\rightarrow 0}|K_r|^{-1}\int_{S} P_w\left(|w-X_{\tau_{}+\gamma_{\delta,H} }| \le  \frac{\eta \delta}{\log{H}}, |X_{2\delta^2}-y|\le r  \right) d\mu(\tau, w) \\&=&
\lim_{r\rightarrow 0}|K_r|^{-1}
\int_{S} \int_{|z|\le \frac{\eta \delta}{ \log{H} }}P_{w-z}\left(|X_{2\delta^2-({\tau_{}+\gamma_{\delta,H}})}-y|\le r   \right)   d\nu_1(z) d\mu(\tau, w)\\ &\ge&
\bigg(1-\frac{C\eta^2}{\sqrt{\log{H}}}\bigg)\int_{S} \int_{|z|\le \frac{\eta \delta}{ \log{H} }}p_{2\delta^2-(\tau_{}+\gamma_{\delta,H})}(w-y)
d\nu_1(z)d\mu(\tau, w)\\ &\ge& \bigg(1-\frac{C\eta^2}{\sqrt{\log{H}}}\bigg)
\int_{S} \int_{|z|\le \frac{\eta \delta}{ \log{H} }}p_{2\delta^2-(\tau_{}+\gamma_{\delta,H})}(w-y)
p_{\gamma_{\delta,H}}(z) dz d\mu(\tau, w)\\ 
&\ge& \bigg(1-\frac{C\eta^2}{\sqrt{\log{H}}}\bigg)^2
\int_{S}p_{2\delta^2-(\tau_{}+\gamma_{\delta,H})}(w-y) d\mu(\tau, w).
\end{eqnarray*}
In the above estimates we used the first and third  of the elementary inequalities        
$$
 \bigg(1-\frac{C\eta^2}{\sqrt{\log{H}}}\bigg)p_{2\delta^2-(\tau_{}+\gamma_{\delta,H})}(w-y) \le p_{2\delta^2-(\tau_{}+\gamma_{\delta,H})}(w-y+z)
$$
$$
\bigg(1-\frac{C\eta^2}{\sqrt{\log{H}}}\bigg)p_{2\delta^2-(\tau_{}+\gamma_{\delta,H})}(w-y+z) \le p_{2\delta^2-(\tau_{}+\gamma_{\delta,H})}(w-y) 
$$
$$
\int_{|z|\le \frac{\eta \delta}{ \log{H} }}p_{\gamma_{\delta,H}}(z) dz \ge 1-\frac{C}{\log{H}}
$$
valid for any $w$ and $\tau$ from the domain of integration, and $z,y,\eta$ as above.

Now set 
\begin{equation*}
{{}\Phi_{V_\beta^A}}=\bigg\{ X \colon
 \,
\int _{0}^{\gamma_{\delta,H} } V_\beta^A(X_u)du  < \zeta E_x\bigg( \int_{0}^{\gamma_{\delta,H} } V_\beta^A(X_u)  du\bigg)       \bigg\}
\end{equation*}
and then put 
\begin{eqnarray*}
{ {\nu_2} (G)=}  P_w \Big( X \colon \, X \in {{}\Phi_{V_\beta^A}}, \,
X_{\gamma_{\delta,H} }\in w+G \Big).
\end{eqnarray*}
It follows from  Lemma \ref{l8} that $\nu_2(\R^d) \le 1-\sigma_0$ so by \eqref{SMP}}
\begin{eqnarray*}
\tilde p(x,y)&=&\mu_{x,y}({{} \Psi_1})=\lim_{r\rightarrow 0}|K_r|^{-1}P_x(\Psi_{1,r})=
\\&=&\lim_{r\rightarrow 0}|K_r|^{-1}\int_{S} \int_{|z|\le \frac{\eta \delta}{ \log{H} }}P_{w-z}\left(|X_{2\delta^2-(\tau_{}+\gamma_{\delta,H})}-y|\le r   \right)d\nu_2(z)d\mu(\tau, w)\\ 
&\le& \bigg(1+\frac{C\eta^2}{\sqrt{\log{H}}}\bigg)\left(1-\sigma_0\right)\int_{S}p_{2\delta^2-(\tau_{}+\gamma_{\delta,H})}(w-y) d\mu(\tau, w).
\end{eqnarray*}
The lemma follows for ${\xi} = 1-\frac{\sigma_0}{2}$ and sufficiently large $H\ge H_{min}(\eta)$.
\end{proof}

Recall that we denote the centre of the cube $Q_{j_s}$ by $q_{j_s}$. {{} Slightly abusing the notation we marginally change the meaning  of 
$\tau_{j_s}$	be setting  by $\tau_{j_s}=\inf\{t\in I_{k_s} \colon \, X_t \in K \cap Q_{j_s}\}$. }
Next, we put 
$$
\Lambda =\bigcup_{1\le H_0\le H} \quad \bigcup_{j_1, \ldots, j_{H_0}, k_1, \ldots, k_{H_0}} {{}\Phi}_{j_1, \ldots, j_{H_0}, k_1, \ldots, k_{H_0} }
$$ 
where we sum over the set of all indices 
$1\le H_0\le H$ and $j_1\ldots, j_{H_0}, k_1, \ldots, k_{H_0}$
such that the system of inequalities 
\begin{eqnarray*}
|X_{2k_s\delta^2}-q_{{j_s}}|\le \eta \delta \sqrt{\log H}\, , \quad \quad  |X_{(2k_s+2)\delta^2}-q_{{j_s}}|\le \eta \delta 
\sqrt{\log H}
\end{eqnarray*}
and 
\begin{eqnarray*}
 |X_{\tau_{j_s}}-X_{\tau_{j_s}+\gamma_{\delta,H} }| \le \frac{\eta \delta} {\log{H}}
\end{eqnarray*}
is not satisfied  for at least one of  $ 1 \le s \le H_0$. 
Then obviously
\begin{eqnarray} \nonumber 
\Lambda \subset \tilde{\Lambda}&=&  \bigg\{ X \colon 
\max_{h \in \{0,1,\ldots,H\}} 
|X_{2h\delta^2}-X_{(2h+2)\delta^2}| > 2\eta \delta \sqrt{\log H} 
 \bigg\} \\ & \bigcup_{} &
\bigg\{ X \colon
\sup_{t_1,t_2\in [0,1], |t_1-t_2|\le \gamma_{\delta,H}} |X_{t_1}-X_{t_2}| > \frac{\eta \delta} {\log{H}}
 \bigg\}. \label{drugi}
\end{eqnarray}
The set $\tilde\Lambda$ is of small probability. In fact we have
\begin{lemma}\label{l10} 
For any $\rho>0$ there exists a constant $\eta$ such that the set $\tilde{\Lambda}$ defined above satisfies the  estimate 
$$
P(\tilde{\Lambda}) \le {H}^{-\rho}
$$
for all sufficiently large $H$. 
\end{lemma}

\begin{proof}
This  lemma follows exactly  in the same way as the proof of H\"older regularity of the Brownian motion.
Directly one can easily check that
$$
P_x\left(|X_{2h\delta^2}-X_{(2h+2)\delta^2}| > \eta \delta \sqrt{\log H}/2\right) \le \exp(-c\eta \log H)
$$
and then sum-up the estimates.
{{} The estimates for the probability of  the second set in \eqref{drugi} follows from the following two observations:

\begin{itemize}
	\item Firstly 
	\begin{eqnarray*}
	\bigg\{ X \colon
	\sup_{t_1,t_2\in [0,1], |t_1-t_2|\le \gamma_{\delta,H}} |X_{t_1}-X_{t_2}| > \frac{\eta \delta} {\log{H}}
	\bigg\} \\
	\subset \bigcup_{0 \le j \le \gamma_{\delta,H}^{-1} }
		\bigg\{ X:
	\sup_{  j \gamma_{\delta,H} \le   t_1  <   (j+1) \gamma_{\delta,H}} |X_{ j \gamma_{\delta,H}}-X_{t_1}| > \frac{\eta \delta} {2\log{H}} 
	 \bigg\}.
	\end{eqnarray*}

	\item Secondly, each event of the RHS of above inclusion has the probability estimated from above. 
	Indeed, if $\{Y_t\}_{t\ge 0}$ is a one dimensional Brownian motion 
	starting at $0$ then by the reflection principle, see for example  \cite[p. 26, formula 4)]{ito}
	$$
		P_x\left( \left\{Y:	\sup_{s<t}{Y_s >a }\right\}\right)=2 	P_x\left(\left\{Y_t >a \right\}\right)
	$$	
	for all $t,a>0$. 
By  Markov property it follows that 	
		\begin{eqnarray*}
		P_x\left(\bigg\{ X:
		\sup_{  j \gamma_{\delta,H} \le   t_1  <   (j+1) \gamma_{\delta,H}} |X_{ j \gamma_{\delta,H}}-X_{t_1}| > \frac{\eta \delta} {2\log{H}} 
		\bigg\}\right)\\
	 \le 
			2P_x\left(\bigg\{ X:
		 |X_{ j \gamma_{\delta,H}}-X_{(j+1) \gamma_{\delta,H}}| > \frac{\eta \delta} {2\sqrt n\log{H}} 
		\bigg\}\right).
	\end{eqnarray*}
	 
\end{itemize}
Then again we sum-up the estimates and 
the lemma follows. }

\end{proof}

In what follow we will also need the following
standard fact about $NTA$ domains

\begin{lemma}\label{har} 
{\rm A)} Assume  $D\subset \R^d$ is an   $NTA$ domain, and let $B_0$, be a closed ball
contained in the interior of $D$ and separated from $K=\partial D$. 
Then, there exists $\gamma >0$, such that the harmonic function on $D\setminus B_0$ vanishing on $K$ and equal to $1$ on 
 $\partial B_0$ satisfies
 \begin{equation}\label{NTA}
h(x)\le C( d_K(x))^\gamma .
\end{equation}
The same statement is valid for $D^c$.

{\rm B)} For dimension  $d=2$ estimate \eqref{NTA} holds both for the domain $D$ 
and its complement $D^c$ with exponent $\gamma =1/2$. 
\end{lemma}
\begin{proof}
We briefly sketch the  proof. We start with Part A. 
Since $D$ satisfies  $NTA$ condition, the Dirichlet problem for $D$ is solvable and 
there exists a function, harmonic in  $D \setminus B_0$ such that  $h(x)=0$ for $x\in K$ and $h(y)=1$ for $y\in \partial B_0$ .
Define $K_j=\{x\in D: d_K(x)= a^j\}$ for sufficiently small fixed $a<1$.
Fix {{} $x\in K_j$. Let $B_1\subset D^c$ be a ball with center $y_0$ and radius  $r$  such that 
$|x-y_0|\le 2 d_K(x)$ 
and, for  some $c$ depending only on the domain $D$
$$
2a^{j}=2 d_K(x) \ge r \ge c d_K(x)=ca^{j}.
$$
Such a ball exists by the $NTA$ conditions for $D$. Now observe that 
$P_x\left(X_{r^2}\in B_1\right)\ge p_0$ and  
$P_x\left(\mbox{there exists $t\le r^2$, such that }X_{t}\in K_{j-1}\right) \le C_1\exp{(-\frac{1}{C_1a})}$ where $p_0,C_1$ depends only on the 
$NTA$ constants of the domain, not on $a$. Consequently, for some  small enough  $a>0$, we have}
\begin{eqnarray*}
P_x\left(\mbox{diffusion starting from 
$K_j$ hits $K$, not hitting $K_{j-1} $ before }\right)\hspace{2cm}{}   \\  \ge
P_x\left(X_{r^2}\in B_1\right)-P_x\left(\mbox{there exists $t\le r^2$, such that }X_{t}\in K_{j-1}\right)  \ge \frac{p_0}{2}
\end{eqnarray*}
for all $j \in \N$.
Hence, any harmonic function bounded by $1$ on $K_{j-1}$ and vanishing on $K$ must be bounded
by $1-\frac{p_0}{2}$ on $K_j$. By a simple induction argument, we obtain $h(x)\le (1-\frac{p_0}{2})^{j-j_0}$
 for  $x\in K_j$, where $j_0$ is the minimal index such that $K_{j_0}$ does not intersect $B_0$.
Now Part A of Lemma \ref{har} follows by the maximum principle.
\medskip 
 
 Part B of the theorem is a consequence of Beurling projection theorem, see \cite[Theorem 3-6, page 43]{Al}. A probabilistic proof  of Beurling projection theorem is described in \cite{Ok}.
 

\end{proof}
{\bf Remark} Note that in the whole paper  we use the NTA conditions only to obtain estimate \eqref{NTA}.
\bigskip

From now on we fix large $\rho$ and the corresponding $\eta=\eta_0 $ given by Lemma \ref{l10}. We assume {{} $H\ge H_{min}(\eta_0)$, where $H_{min}(\eta_0)$ is chosen in the same way as in Lemma \ref{l9}.} 

In the next lemma we estimate the probability of the set of paths for which  the number $H_0$ is small.  
\begin{lemma}\label{l11} 
There exist  constant $C$ such that for any $\kappa''>\kappa' >0$ the following estimate holds
$$
P_x\left(X: \#\{h: X \mbox{ hits the boundary for $t\in I_h$}\} <H^{\kappa'} \right)< CH^{\kappa''}\delta^{\theta}.
$$
where $\theta=2\gamma$ and  $\gamma$ is the exponent from  Lemma \ref{har}.
\end{lemma}
\begin{proof}
Let ${{}\Phi}_{j_1, \ldots, j_{H_0}, k_1, \ldots, k_{H_0}, 0 }={{}\Phi}_{j_1, \ldots, j_{H_0}, k_1, \ldots, k_{H_0}} \setminus \Lambda$. Note that 
\begin{eqnarray*}
{{}\Phi}_{j_1, \ldots, j_{H_0}, k_1, \ldots, k_{H_0}, 0 } \subset {{}\widetilde{\Phi}}_{j_1, \ldots, j_{H_0}, k_1, \ldots, k_{H_0} }
\end{eqnarray*}
where 
\begin{eqnarray*}
\leftline{$ {{}\widetilde{\Phi}}_{j_1, \ldots, j_{H_0}, k_1, \ldots, k_{H_0} }=$}\\=\left\{X \colon \mbox{ $X_{t}=\tilde{X_{t}}$ for all $t\le (2k_{H_0}+2)\delta^2$ and some $\tilde{X}\in {{}\Phi}_{j_1, \ldots, j_{H_0}, k_1, \ldots, k_{H_0}, 0 }$	}\right\}.
\end{eqnarray*}
Next, observe that for a fixed $H_0\ge 1$, the sets $ {{}\widetilde{\Phi}}_{j_1, \ldots, j_{H_0}, k_1, \ldots, k_{H_0}} $ are mutually disjoint. Let $\theta = 2\gamma$, ($\gamma$ defined by \eqref{NTA}).
We will prove the estimate
\begin{eqnarray*}
P_x\left( {{}\Phi}_{j_1, \ldots, j_{H_0}, k_1, \ldots, k_{H_0}, 0 } \right)  \le 
C_{\theta} \delta^{\theta} P_x\left( {{}\widetilde{\Phi}}_{j_1, \ldots, j_{H_0}, k_1, \ldots, k_{H_0}} \right).
\end{eqnarray*}
Next, we  define the measure  $\nu_3$ by the formula 
$$\nu_3( {{}G})=P_x\left(X\in  {{}\widetilde{\Phi}}_{j_1, \ldots, j_{H_0}, k_1, \ldots, k_{H_0}}: X_{ (2k_{H_0}+2)\delta^2}\in {{}G}\right).$$
By the Markov property of Brownian motion we have
\begin{eqnarray*}
&P_x&\!\!\!\Big({{}\Phi}_{j_1, \ldots, j_{H_0}, k_1, \ldots, k_{H_0}, 0 }\Big) \ \\
& =&\int P_x\left(X_{1- (2k_{H_0}+2)\delta^2}\in B_0 \mbox{ and } X_s\cap K =\emptyset
\mbox{ for $0\le s \le 1- (2k_{H_0}+2)\delta^2 $}   \right) d\nu_3(x)
\end{eqnarray*}
where by the definition of $ {{}\Phi}_{j_1, \ldots, j_{H_0}, k_1, \ldots, k_{H_0} }$, for $x\in \mbox{supp} 
\{d\nu_3\}$  we have $$|x-q_{k_{H_0}}|\le \delta^2 \sqrt{\log{H}}.  $$ 
By  Lemma \ref{har} 
\begin{eqnarray*}
 P_x\left(X_{1- (2k_{H_0}+2)\delta^2}\in B_0 \mbox{ and } X_s\cap K =\emptyset
\mbox{ for $0\le s \le 1- (2k_{H_0}+2)\delta^2 $}   \right)\\  \le
 P_x\Big( X: X \mbox{ hits first time into $B_0$, not into $K$}  \Big)\\
 \le  C\left(\delta^2 \sqrt{\log{H}}\right)^\gamma {}\hspace{2cm}
\end{eqnarray*}
and consequently, since $\nu_3 (\R^d) = P_x\left( {{}\widetilde{\Phi}}_{j_1, \ldots, j_{H_0}, k_1, \ldots, k_{H_0}} \right)$ it follows that 
$$
P_x\left({{}\Phi}_{j_1, \ldots, j_{H_0}, k_1, \ldots, k_{H_0}, 0 }\right)\le  
C\left(\delta^2 \sqrt{\log{H}}\right)^\gamma  P_x\left( {{}\widetilde{\Phi}}_{j_1, \ldots, j_{H_0}, k_1, \ldots, k_{H_0}} \right).
$$

Hence, since ${{}\widetilde{\Phi}}_{j_1, \ldots, j_{H_0}, k_1, \ldots, k_{H_0}}$
are mutually disjoint, we obtain 
\begin{eqnarray*}
\sum_{1\le H_0 \le H^{\kappa'}}\quad \sum_{j_1, \ldots, j_{H_0}, k_1, \ldots, k_{H_0}}P_x\left( {{}\Phi}_{j_1, \ldots, j_{H_0}, k_1, \ldots, k_{H_0}, 0} \right) 
 \le C H^{\kappa''}\delta^{\theta}.
\end{eqnarray*}
Now the lemma follows from the above estimates and  Lemma \ref{l10}.

\end{proof}
Now we are ready to prove   Theorem \ref{main1}. 

\medskip

{\it Proof of  Theorem \ref{main1}.}
{{} Our aim is to apply Feynman-Kac formula \eqref{FK}.
	Observe that without loss of generality, in \eqref{FK}  we can consider the expectation  over  the set of trajectories  
 which }
hit $K$ at a time  $\tau \in I_h $ for some $h$. Indeed,  if the path does not cross $K$ at any  $t\in I_h$ for any $h$,
then it must cross $K$ at a time  $\tau \in J_h$  for some $h$, and we repeat the  argument  replacing $I_h$ by $J_h$.

Now consider the set 
\begin{eqnarray*}
\Phi_{small}=\bigcup_{H_0 \le  H^{\kappa'}} \quad  \bigcup_{j_1, \ldots, j_{H_0}, k_1, \ldots, k_{H_0}} {{}\Phi}_{j_1, \ldots, j_{H_0}, k_1, \ldots, k_{H_0}, 0 }
\end{eqnarray*}
Recall, that we   set  $\delta \approx A^{-v}$, $H\approx \delta^{-2}$ so by Lemma \ref{l11}
\begin{equation}\label{small}
\int_{{{}\Phi}_{small}}e_{V_\beta^A}(t)\chi_{B_0}(X_t) P_x(dX)\le P_x(\Phi_{small})  \le C H^{\kappa''}\delta^{\theta} \le 
CA^{-2v(\gamma -\kappa'')}.
\end{equation}

Next,  consider the following sets 
\begin{eqnarray*}
\leftline{$ {{}\Phi}_{j_1, \ldots, j_{H_0}, k_1, \ldots, k_{H_0}, 0 }^{bad}=  {{}\Phi}_{j_1, \ldots, j_{H_0}, k_1, \ldots, k_{H_0}, 0 }$} \\ \bigcap \left\{ \int _{\tau_{j_s}}^{\tau_{j_s}+\gamma_{\delta,H} } V_\beta^A(X_s)ds  \le  \zeta E\left( \int_{\tau_{j_s}}^{\tau_{j_s}+\gamma_{\delta,H} } V_\beta^A(X_s)\right) ds,   \, \mbox{on every $I_{k_s}$}\right\}\,;
\end{eqnarray*}
$${{}\Phi}_{essential} =\bigcup_{H_0> H^{\kappa'}} \quad  \bigcup_{j_1, \ldots, j_{H_0}, k_1, \ldots, k_{H_0}} {{}\Phi}_{j_1, \ldots, j_{H_0}, k_1, \ldots, k_{H_0}, 0 }\,;$$ 
$${{}\Phi}_{bad} =\bigcup_{H_0> H^{\kappa'}} \quad  \bigcup_{j_1, \ldots, j_{H_0}, k_1, \ldots, k_{H_0}}{{}\Phi}_{j_1, \ldots, j_{H_0}, k_1, \ldots, k_{H_0}, 0 }^{bad}\,;$$
$${{}\Phi}_{good} ={{}\Phi}_{essential}\setminus{{}\Phi}_{bad}.$$
 Note that $\gamma_{\delta,H}=\delta^2(\log{H})^{-3}\approx A^{-2v}(\log{H})^{-3} $. 
Let $2vd<\beta +\alpha -d$.  
If $X\in {{}\Phi}_{good}$ then, for at least one $j_s$ we have 
\begin{eqnarray*}
 \int_{\tau_{j_s}}^{\tau_{j_s}+\gamma_{\delta,H} } V_\beta^A(X_s)ds  \ge \zeta E_x\left( \int_{\tau_{j_s}}^{\tau_{j_s}+\gamma_{\delta,H}} V_\beta^A(X_s)ds\right) \ge C A^{\beta+\alpha-d}\frac{\delta^{2d}}{(\log H)^{3d}}
 \\ \ge C\frac{A^{\beta+\alpha-d-2dv}}{(\log H)^{3d}}\ge C A^{\iota}
\end{eqnarray*}
for any $\iota<\beta+\alpha-d-2dv$ and  for some constant  $C>0$ 
depending only on the domain and $\gamma$, but not on $A$. 
Hence, for any fixed $\rho>0$ and sufficiently large $A$ we have 
\begin{equation}\label{gq}
\int_{{{}\Phi}_{good}}e_{V_\beta^A}(t)\chi_{B_0}(X_t) P_x(dX)\le \exp(-CA^\iota)    \le C' A^{-\rho}.
\end{equation}
Next we shall  prove that for any $\rho>0$ we have  $P({{}\Phi}_{bad})\le CA^{-\rho}$.
To that end we  will show that 
\begin{equation}\label{tilde}
P_x\left({{{}\Phi}}^{bad}_{j_1, \ldots, j_{H_0}, k_1, \ldots, k_{H_0}}\right)\le {\xi}^{H_0}P_x\Big({{}\Phi}_{j_1, \ldots, j_{H_0}, k_1, \ldots, k_{H_0}}\Big),
\end{equation}
where ${\xi}$ is the constant from Lemma \ref{l9}.

	We fix some notation. For $\epsilon \in \{0,1 \}$ and for $s \in {1,\ldots,H_0}$ we denote $p_{k_s}^{\epsilon}(x,y)$
simply setting     $p_{k_s}^{0}(x,y)=p(x,y)$ and $p_{k_s}^{1}(x,y)=\tilde{p}(x,y)$
where $p$ and $\tilde{p}$ are defined by \eqref{p} and \eqref{pt}.
 Secondly,  for $h \in \{1,\ldots, H\} \setminus  \{k_1,\ldots,k_{H_0}\}$ we denote
$p_{h}^{0}(x,y)$ as the  Brownian bridge measure $d\mu_{x,y}$ of 
\begin{eqnarray}
p_{h}^{0}(x,y)=\mu_{x,y} \Big(X\in \Omega_{x,y}^{2\delta^2} \colon \,  X_t \notin  K \quad \forall{t\in[0,\delta^2]}  \, \Big)
 \nonumber
\end{eqnarray}

 We put $\epsilon_h=1$ if $h \in \{k_1,\ldots,k_{H_0}\}$  $\epsilon_h=0$  for $h \in \{1,\ldots, H\} \setminus  \{k_1,\ldots,k_{H_0}\}$.  Now, by Lemma  \ref{l9}
	$p_h^1(x,y) \le {\xi} p_k^0(x,y)$
	for $h \in \{k_1,\ldots,k_{H_0}\}$ so by the 
	  Markov property: 

\begin{eqnarray*}
P_x\Big({{{}\Phi}}^{bad}_{j_1, \ldots, j_{H_0}, k_1, \ldots, k_{H_0}}\Big)=\\
\int_{B_0}\int p_{0}^{\epsilon_0}(x,x_1)  p_{1}^{\epsilon_1}(x_{1},x_{2})
 \ldots  p_{H-1}^{\epsilon_{H-1}}(x_{H-1},y)dx_1\ldots dx_{H-1}dy \le 
 \\ 
 \le {\xi}^{H_0} 
\int_{B_0} \int p_{0}^{0}(x,x_1)  p_{1}^{0}(x_{1},x_{2})
 \ldots  p_{H-1}^{0}(x_{H-1},y)dx_1\ldots dx_{H-1}dy \\
 ={\xi}^{H_0}P_x\Big({{}\Phi}_{j_1, \ldots, j_{H_0}, k_1, \ldots, k_{H_0}}\Big).
\end{eqnarray*}
 The above inequality  proves \eqref{tilde}.
 Now, since the sets ${{{}\Phi}}^{}_{j_1, \ldots, j_{H_0}, k_1, \ldots, k_{H_0}}$ are disjoint, 
  for  $\delta \approx A^{-v}$ and any $C>0$ we have
 \begin{eqnarray}
 P_x\Big({{{}\Phi}}_{bad}\Big)\le
\sum_{H_0\ge H^{\kappa'}}\sum_{j_1, \ldots, j_{H_0}, k_1, \ldots, k_{H_0}}P_x\Big({{{}\Phi}}^{bad}_{j_1, \ldots, j_{H_0}, k_1, \ldots, k_{H_0}}\Big)\le \label{lab_bed_2}\\
\le C {\xi}^{H^{\kappa'}}\sum_{H_0\ge H^{\kappa'}}\sum_{j_1, \ldots, j_{H_0}, k_1, \ldots, k_{H_0}}P_x\Big({{{}\Phi}}^{}_{j_1, \ldots, j_{H_0}, k_1, \ldots, k_{H_0}}\Big)  \le C_{v, C} A^{-\rho}\nonumber
\end{eqnarray}
for sufficiently large $H$. 
Combining the  estimates \eqref{small},   \eqref{gq} and    \eqref{lab_bed_2}  we get that if  $0<\beta+\alpha-d-2dv$ then for 
any $\rho < \gamma \frac{\beta+\alpha -d}{d}$ and $\kappa'',\kappa'$ small enough 
\begin{eqnarray*}
 \int_{\Phi_{small}\cup \Phi_{good} \cup \Phi_{bad}   } e_{V_\beta^A}(t)\chi_{B_0}(X_t) P_x(dX)  
{}\hspace{3cm}
\\ \le 
CA^{-2v(\gamma -\kappa'')}+CA^{-\rho}+ P_x\Big({{{}\Phi}}_{bad}\Big) \le CA^{-v\gamma}.
\end{eqnarray*}
This ends the proof.  

\hfill $\Box$

\bigskip

\noindent
{\bf Acknowledgements:}  Adam Sikora was partially supported by an 
Australian Research Council (ARC) Discovery Grants  DP130101302 and DP160100941.
Jacek Zienkiewicz  was partially supported by  NCN
grant   UMO-2014/15/B/ST1/00060.  We would like to thank the anonymous referee for carefully reading our manuscript and
making several valuable comments and suggestions.

\end{document}